\documentclass[12pt]{article}
\usepackage{amssymb}
\usepackage{hyperref}
\hypersetup{colorlinks=false, pdfborder={0 0 0}}
\newtheorem{theorem}{Theorem}
\newtheorem{unnumber}{}

\newtheorem{prepf}[unnumber]{Proof}
\newenvironment{proof}{\prepf\rm}{\endprepf}
\newcommand{\qed}{\hfill$\Box$}
\begin{document}
\title{Undirecting membership in models of ZFA}
\author{Bea Adam-Day%
  \thanks{University of Leeds, email: \url{B.Adam-Day@leeds.ac.uk} \textsc{orcid}: \url{https://orcid.org/0000-0002-7891-916X}} \and Peter J. Cameron%
\thanks{University of St Andrews, email: \url{pjc20@st-andrews.ac.uk} \textsc{orcid:} \url{https://orcid.org/0000-0003-3130-9505}}}
\date{}
\maketitle
\begin{abstract}
It is known that, if we take a countable model of Zermelo--Fraenkel set theory
ZFC and ``undirect'' the membership relation (that is, make a graph by joining
$x$ to $y$ if either $x\in y$ or $y\in x$), we obtain the Erd\H{o}s--R\'enyi
random graph. The crucial axiom in the proof of this is the Axiom of
Foundation; so it is natural to wonder what happens if we delete this axiom,
or replace it by an alternative (such as Aczel's Anti-Foundation Axiom).
The resulting graph may fail to be simple; it may have loops (if $x\in x$ for
some $x$) or multiple edges (if $x\in y$ and $y\in x$ for some $x,y$). We
show that, in ZFA, if we keep the loops and ignore the multiple edges, we
obtain the ``random loopy graph'' (which is $\aleph_0$-categorical and
homogeneous), but if we keep multiple edges, the resulting graph is not
$\aleph_0$-categorical, but has infinitely many $1$-types. Moreover, if we
keep only loops and double edges and discard single edges, the resulting
graph contains countably many connected components isomorphic to any given
finite connected graph with loops.
\end{abstract}
\section{Introduction}

According to the downward L\"owenheim--Skolem
theorem~\cite[Corollary 3.1.4]{hodges}, if a
first-order theory in a countable language is consistent, then it has a
countable model. In particular, Zermelo--Fraenkel set theory ZFC, if
consistent, has a countable model. (This is the source of the \emph{Skolem
paradox}, since the existence of uncountable sets is a theorem of ZFC.) Indeed
there are many different countable models, but they all have a common feature.
To describe this, we briefly introduce the Erd\H{o}s--R\'enyi random graph
(sometimes referred to as Rado's graph, for reasons we will see).

Erd\H{o}s and R\'enyi~\cite{er} showed that there is a countable graph $R$ such
that, if a random graph $X$ on a fixed countable vertex set is chosen by
selecting edges independently with probability $1/2$ (or, indeed, any fixed
$p$ with $0<p<1$), then $X$ is isomorphic to $R$ almost surely (that is, with
probability $1$). Moreover, $R$ is highly symmetric; they showed that such a
graph has infinitely many automorphisms, but in fact it is \emph{homogeneous}:
any isomorphism between finite induced subgraphs extends to an automorphism.
Erd\H{o}s and R\'enyi gave a non-constructive existence proof, based on the
following property, called the \emph{Alice's Restaurant property}, or ARP:
\begin{quote}
Given any two disjoint finite sets $U$ and $V$ of vertices, there is a vertex 
$z$ joined to every vertex in $U$ and no vertex in $V$.
\end{quote}
In other terminology, $R$ is the \emph{Fra\"{\i}ss\'e limit} of the class of
finite graphs. For further discussion of the graph $R$, we refer to
\cite{c:rg}; for Fra\"{\i}ss\'e's theorem, see \cite[Theorem 6.1.2]{hodges}.

A model of Zermelo--Fraenkel set theory ZF consists of a collection of objects
called sets, and a binary relation $\in$ on this collection. In other words,
it is a directed graph. By ``undirecting'' this relation $\in$, that is,
defining an undirected graph in which $x$ and $y$ are adjacent if either
$x\in y$ or $y\in x$, we obtain a simple graph.  (At the end of this section,
we state the Axiom of Foundation, one of the axioms of ZF, and show that it
forbids directed cycles for the membership relation; in particular it forbids
loops ($x\in x$) and double edges $(x\in y\in x$) in the undirected graph.) We
call this the \emph{membership graph} of the model.

\begin{theorem}
The membership graph of a countable model of ZFC is isomorphic to the
Erd\H{o}s--R\'enyi random graph $R$.
\label{t1}
\end{theorem}

\begin{proof}
We verify the ARP. Let $U$ and $V$ be finite disjoint sets of vertices. Take
$z=U\cup\{V\}$. (The existence of $z$ is guaranteed by the Pairing and Union
axioms.) If $u\in U$, then $u\in z$, so $z$ is joined to $u$.
Suppose, for a contradiction, that $z$ is joined to a vertex $v\in V$. There
are two cases:
\begin{itemize}\itemsep0pt
\item $v\in z$. Since $v\notin U$, we must have $v=V$, so $v\in v$, 
contradicting Foundation.
\item $z\in v$. Then $v\in V\in z\in v$, also contradicting Foundation.\qed
\end{itemize}
\end{proof}

\paragraph{Remark} Observe that, in the proof, we use only the Empty Set axiom
(asserting that sets exist), the Pairing axiom, and the Axiom of Foundation.
The other axioms (Infinity, Selection, Choice, and so on) are not required.

In particular, there is a standard model of ZFC with the negation of the Axiom
of Infinity, or \emph{hereditarily finite set theory} HF. We take the sets to
be the natural numbers, and represent a finite set $\{a_1,\ldots,a_n\}$ of
natural numbers by $b=2^{a_1}+\cdots+2^{a_n}$, so that, for $a<b$, we have $a$
joined to $b$ if and only if the $a$th digit in the base~$2$ expansion of $b$
is $1$. The model of $R$ given by undirecting this membership relation is
precisely the graph constructed by Rado~\cite{rado}.

\paragraph{The Axiom of Foundation}

We state the Axiom of Foundation, and its role in forbidding directed cycles
for the membership relation in models of ZF; see  \cite{c:slc,devlin} for more
details.

The Axiom of Foundation states:
\begin{quote}
For every non-empty set $x$, there exists $y\in x$ such that
$x\cap y=\emptyset$.
\end{quote}
Suppose that there were a directed cycle
$x_0\in x_1\in\cdots\in x_{n-1}\in x_0$. Let $x=\{x_0,x_1,\ldots,x_{n-1}\}$.
For any $y\in x$, say $y=x_i$, we have $x_{i-1}\in x\cap y$, contradicting
the Axiom of Foundation.

In fact the axiom also forbids infinite descending chains for the membership
relation, and is ``equivalent'' to this (but not by a first-order implication
since there is no first-order formula forbidding such chains).

\section{The Anti-Foundation Axiom}

Since the Axiom of Foundation is required for the proof of Theorem~\ref{t1},
what happens if we delete it, or replace it by an alternative? We consider this
question when Foundation is replaced by Aczel's Anti-Foundation Axiom. 
Following Barwise and Moss~\cite{bm}, we use this axiom in the form of the
Solution Lemma \cite[p.72]{bm}, which we briefly discuss.

Let $X$ be a set of ``indeterminates'', and $A$ a set of sets called ``atoms''.
A \emph{flat system of equations} is a set of equations of the form $x=S_x$,
where $S_x$ is a subset of $X\cup A$ for each $x\in X$. A \emph{solution}
to the system is an assignment of sets to the indeterminates so that the
equations become true.

For example, if $A=\{a,b\}$, then
\begin{eqnarray*}
x &=& \{y,a\}, \\
y &=& \{x,b\}
\end{eqnarray*}
is a flat system of equations.

The \emph{anti-foundation axiom}, or AFA, asserts that any flat system of
equations has a \emph{unique} solution.

Note that the solution to the above system will satisfy $x\in y$ and $y\in x$,
so will correspond to a double edge in the membership graph. Similarly, the
solution to $x=\{x\}$ will give a loop in the graph.

The axioms system ZFA denotes ZFC with the Axiom of Foundation deleted and
replaced by the axiom AFA. Our concern is with membership graphs of models of
ZFA. We note in passing that, if ZFC is consistent then so is ZFA: see
\cite[Chapter~9]{bm}.

Note that Barwise and Moss work in a set theory containing ``urelements''
which are not themselves sets; this makes no difference to our arguments.

\section{Membership graphs of models of ZFA}

The argument showing that the membership graph of a model of ZFC is a
simple graph does depend on Foundation, as we saw. In ZFA, we have sets
$x$ with $x\in x$, giving loops in the graph; and pairs $x,y$ of sets with
$x\in y\in x$, giving double edges.

The \emph{random loopy graph} is obtained with probability $1$ if we choose
a graph on a countable vertex set by choosing edges (including loop edges)
from pairs of not necessarily distinct vertices independently with probability
$1/2$. It is homogeneous, and is the Fra\"{\i}ss\'e limit of the class of
finite loopy graphs. The relevant version of the Alice's Restaurant property
characterises it as the countable graph such that, for any two finite disjoint
sets $U$ and $V$ of vertices, there are vertices $z_1$ and $z_2$, where $z_1$
is loopless and $z_2$ has a loop, each joined to all vertices in $U$ and to
none in $V$. The proof is very similar to the usual proof for the random graph,
and we will not give it here.

\begin{theorem}
The membership graph of a countable model of ZFA, ignoring multiple edges but
keeping loops, is isomorphic to the random loopy graph.
\label{t2}
\end{theorem}

\begin{proof}
We begin with some preliminaries. In a model of a subset
of ZF including at least Selection, there is no set whose members are all
sets. For, if such a set $S$ exists, then Selection would give Bertrand
Russell's set $R=\{x\in S:x\notin x\}$, whose existence leads to a contraction
on examining whether $R\in R$ or not.

It follows that, if the Union axiom also holds, there is no set $S'$ which
contains all the $p$-element sets, for a fixed positive natural number
$p$: for the union of $S'$ would be $S$. In particular, if $T$ is any set,
then there is a set of cardinality $p$ which is not a member of $T$.

Now let $\Gamma$ be the membership graph of a countable model of ZFA, with
loops but no double edges. We will show that $\Gamma$ satisfies the loopy
version of ARP. For this let $\{u_1,\ldots,u_m\}$ and $\{v_1,\ldots,v_n\}$ be
disjoint sets, and let
$z_1 = \{x, u_1, \ldots, u_m \}$ and
$z_2 = \{z_2, x, u_1, \ldots , u_m\}$,
where $x$ is a vertex satisfying the following conditions:
\begin{itemize}\itemsep0pt
\item $x$ is not equal to any of the $v_j$;
\item $x$ is not contained in any of the $v_j$ or the $u_i$;
\item $x$ is not contained in any of the sets contained in any of the $v_j$.
\end{itemize}
Such an $x$ exists, since otherwise the union of $V$, $\bigcup U$, 
$\bigcup V$, and $\bigcup\bigcup V$ would contain every set, a contradiction.

Furthermore, we may assume that $|x|=m+3$, since by our earlier remarks there
is a set of this cardinality not a member of the ``forbidden set''
$U\cup V\cup(\bigcup V)\cup(\bigcup\bigcup V)$ above.

We remark that the existence of $z_1$ follows simply from Pairing and Union;
for $z_2$, we invoke Anti-Foundation, letting $z_2$ be the unique solution of
the equation
\[z = \{z, x, u_1, \ldots , u_m\}.\]

Both $z_1$ and $z_2$ are joined to all the vertices $u_i$; and by construction,
there is a loop on $z_2$. We claim that there is no loop on $z_1$. For such
a loop would imply one of the following:
\begin{itemize}\itemsep0pt
\item $z_1=u_i$ for some $i$. Then we have $x\in u_i$, contradicting our
choice of $x$.
\item $z_1=x$. But we chose $x$ with $|x|=m+3$, whereas $|z_1|\le m+1$. (Note
that in the same way we see that $z_2\ne x$.)
\end{itemize}

Finally we have to show that $z_1$ and $z_2$ are not joined to any $v_j$.
We cannot have any $v_j$ contained in $z_1$ or $z_2$; for the $v_j$
are distinct from the $u_i$ by hypothesis, not equal to $x$ by choice of $x$,
and not equal to $z_2$ since if so then $x$ would be a member of $v_j$, again
contrary to our choice of $x$. Also we cannot have $z_1$
or $z_2\in v_j$, since if so then $x$ belongs to a member of $v_j$, again
contrary to our choice of $x$.\qed
\end{proof}

\paragraph{Remark} As in the case of ZFC, it is interesting to note which
axioms are actually used in the proof. The Empty Set, Pairing and Union axioms
are once again used; of course, the Anti-Foundation Axiom is used; and as well,
we use the Selection Axiom.

\medskip

What happens if we keep the multiple edges? We cannot describe all graphs that
can arise, but we note that there is no such graph which has the nice
properties of $\aleph_0$-categoricity and homogeneity which hold in the random
and random  loopy graphs. This will follow from the theorem of Engeler,
Ryll-Nardzewski and Svenonius (see~\cite[Theorem 6.3.1]{hodges}), according to
which a countable first-order structure is $\aleph_0$-categorical (that is,
determined uniquely by its first-order theory and the property of countability)
if and only if its theory has only finitely many $n$-types for all $n$.

\begin{theorem}
The membership graph of a countable model of ZFA, keeping double edges, is not
$\aleph_0$-categorical.
\label{t3}
\end{theorem}

\begin{proof}
Take a countable model of ZFA.
Let $a_n$ be distinct well-founded sets for $n\in\mathbb{N}$, for example, the
natural numbers. For every natural number $n$, consider the equations
\begin{eqnarray*}
y &=& \{x_0,\ldots,x_{n-1}\},\\
x_i &=& \{y,a_i\} \hbox{ for }i=0,\ldots,n-1.
\end{eqnarray*}
By AFA, these equations have a unique solution in the model.
We have $x_i\in y$ and $y\in x_i$, so all the edges $\{x_i,y\}$
are double. (These sets are all distinct, by extension.) There are no further
double edges on $y$, since if $\{y,z\}$ is a double edge then $z\in y$ and
so $z=x_i$ for some $i$.

Thus, for every natural number $n$, there is a set lying on exactly $n$ double
edges. So there are infinitely many $1$-types in the graph, and it cannot be
$\aleph_0$-categorical, by the theorem of Engeler, Ryll-Nardzewski and
Svenonius.

Moreover, we can take the infinite set of equations
\begin{eqnarray*}
y &=& \{x_n:n\in\mathbb{N}\}, \\
x_n &=& \{y,a_n\} \hbox{ for }n\in\mathbb{N}.
\end{eqnarray*}
A solution to these equations will be a point lying on infinitely many double
edges. \qed
\end{proof}

Another natural reduct is obtained by keeping only the double edges.
The \emph{double-edge graph} of a model of ZFA has as vertices the sets and
as edges all pairs $\{x,y\}$ with $x\in y$ and $y\in x$ (allowing $x=y$).
Thus, it includes loops and double edges but omits all ``conventional''
instances of the membership relation (where $a\in b$ but $b\notin a$).

\begin{theorem}
Let $D$ be the double-edge graph of a countable model of ZFA. Then, for any
finite connected loopy graph $\Gamma$, $D$ has infinitely many connected
components isomorphic to $\Gamma$. It also has at least one infinite component.
\label{t4}
\end{theorem}

\begin{proof}
An example will illustrate the general proof. Let $\Gamma$ be the $4$-cycle
with edges $\{v_0,v_1\}$, $\{v_1,v_2\}$, $\{v_2,v_3\}$ and $\{v_3,v_0\}$,
together with a loop at $v_0$. Take any four well-founded sets $a_0,a_1,
a_2,a_3$ (for example, the first four natural numbers), and consider the
equations
\begin{eqnarray*}
y_0 &=& \{a_0,y_0,y_1,y_3\},\\
y_1 &=& \{a_1,y_0,y_2\},\\
y_2 &=& \{a_2,y_1,y_3\},\\
y_3 &=& \{a_3,y_0,y_2\}.
\end{eqnarray*}
The unique solution gives an induced subgraph isomorphic to $\Gamma$. Note
that there are no other double edges meeting these vertices: if, say,
$\{y_1,x\}$ were a double edge, then $x\in y_1$, and so by Extensionality,
$x=a_1$ or $x=y_0$ or $x=y_2$; the first is impossible since $y_1\notin a_1$
by assumption. So the given set is a connected component.

Since there are infinitely many possible choices of $a_0,\ldots,a_3$, there
are infinitely many such connected components.

We saw earlier that there is a vertex with infinite valency; it lies in an
infinite component of $D$. \qed
\end{proof}

This leaves a few questions which we have not been able to answer in this current work.
\begin{enumerate}\itemsep0pt
\item Is it true that the first-order theory of the membership graph of a
countable model of ZFA has infinitely many countable models?
\item Is it true that there are infinitely many non-isomorphic graphs which are
membership graphs of countable models of ZFA?
\item Can more be said about infinite connected components of the double-edge
graph?
\item What about models of ZFA where the Axiom of Infinity is replaced with
its negation?
\item\label{last} Is it true that, if two countable multigraphs are
elementarily equivalent, and one is the membership graph of a model of ZFA,
then so is the other?
\end{enumerate}
However, in a sequel by the first author, John Howe and Rosario Mennuni~\cite{adhm} the first two questions are answered affirmatively and a characterisation of the connected components of double-edge graphs is given, thus answering the third question. The analogue to Question~\ref{last} for double-edge graphs is shown to be negative, and indeed it is shown that, for any double-edge graph, there is an elementarily equivalent countable structure which is not itself a double-edge graph.

\end{document}